\documentclass[12pt]{amsart}
\usepackage{bbm}
\usepackage{bbm}
\usepackage{amsfonts}
\usepackage{stmaryrd}
\usepackage{}
\usepackage{amsfonts}      
\usepackage{txfonts}
\usepackage{amssymb}
\usepackage{eucal}
\usepackage{graphicx}
\usepackage{amsmath}
\usepackage{amscd}
\usepackage[all]{xy}           

\usepackage{amsfonts,latexsym}
\usepackage{xspace}
\usepackage{epsfig}
\usepackage{float}
\usepackage{color}
\usepackage{fancybox}
\usepackage{colordvi}
\usepackage{multicol}
\usepackage{colordvi}
\usepackage{stmaryrd}
\usepackage[normalem]{ulem}
\usepackage[colorlinks,final,backref=page,hyperindex,hypertex]{hyperref}
\usepackage[active]{srcltx} 

\newcommand\changed[1]{#1}
\changed{}

\newtheorem{theorem}{Theorem}[section]
\newtheorem{lemma}[theorem]{Lemma}
\newtheorem{coro}[theorem]{Corollary}

\newtheorem{prop}[theorem]{Proposition}
\theoremstyle{definition}
\newtheorem{defn}[theorem]{Definition}

\newtheorem{exam}[theorem]{Example}
\newtheorem{claim}[theorem]{Claim}

\newcommand{\nc}{\newcommand}
\nc{\tred}[1]{\textcolor{red}{#1}} \nc{\tblue}[1]{\textcolor{blue}{#1}} \nc{\tgreen}[1]{\textcolor{green}{#1}} \nc{\tpurple}[1]{\textcolor{purple}{#1}} \nc{\btred}[1]{\textcolor{red}{\bf #1}} \nc{\btblue}[1]{\textcolor{blue}{\bf #1}} \nc{\btgreen}[1]{\textcolor{green}{\bf #1}} \nc{\btpurple}[1]{\textcolor{purple}{\bf #1}}

\renewcommand{\Bbb}{\mathbb}

\newcommand{\delete}[1]{}
\newcommand{\mynotes}[1]{}
\newcommand\notes[1]{}

\nc{\mlabel}[1]{\label{#1}}  
\nc{\mcite}[1]{\cite{#1}}  
\nc{\mref}[1]{\ref{#1}}  
\nc{\mbibitem}[1]{\bibitem{#1}} 

\delete{
\nc{\mlabel}[1]{\label{#1}  
{\hfill \hspace{1cm}{\bf{{\ }\hfill(#1)}}}}
\nc{\mcite}[1]{\cite{#1}{{\bf{{\ }(#1)}}}}  
\nc{\mref}[1]{\ref{#1}{{\bf{{\ }(#1)}}}}  
\nc{\mbibitem}[1]{\bibitem[\bf #1]{#1}} 
}

\renewcommand\geq{\geqslant}
\renewcommand\leq{\leqslant}

\renewcommand\bar[1]{\overline{#1}}
\renewcommand\tilde[1]{\widetilde{#1}}

\nc{\vs}{\vec{\star}\xspace}
\nc{\occ}{occurrence\xspace}
\nc{\occs}{occurrences\xspace}
\nc{\Occ}{Occurrence\xspace}
\nc{\pla}{placement\xspace}
\nc{\plas}{placements\xspace}
\nc{\Pla}{Placement\xspace}
\nc{\Plas}{Placements\xspace}

\nc{\loc}{location\xspace}
\nc{\locs}{locations\xspace}
\nc{\phs}{\phi_{\star}}
\nc{\phsi}{\phi_{\star}^{-1}}
\nc{\pts}{\phi_{\star_1,\star_2}}
\nc{\ptsi}{\phi_{\star_1,\star_2}^{-1}}
\nc{\lm}{\,\llfloor}
\nc{\rtm}{\rrfloor\,}
\nc{\mo}{\mathbf o}
\nc{\pl}{\mathfrak{p}}

\nc{\rbw}{\mathfrak{R}} \nc{\brp}{\mathrm{brp}} \nc{\lead}{\mathrm{Lead}} \nc{\Id}{\mathrm{Id}} \nc{\Irr}{\mathrm{Irr}} \nc{\vx}{\sigma} \nc{\vy}{\tau} \nc{\dvx}{\sigma^{(1)}} \nc{\dvy}{\tau^{(1)}} \nc{\done}{\vep} \nc{\citep}[1]{\cite{#1}} \nc{\wt}{\mathrm{wt}} \nc{\bre}[1]{|#1|} \nc{\mapmonoid}{\frakM} \nc{\disjoint}{\frakM'}
\nc{\ncpoly}[1]{\langle #1\rangle}  
\nc{\mapm}[1]{\lfloor\!|{#1}|\!\rfloor}
\nc{\diff}[1]{{}^\NC\{ #1 \}} \nc{\disj}[1]{\{{#1}\}'} \nc{\mdisj}[1]{\frakM'(#1)} \nc{\brho}{\bar{\rho}} \nc{\om}{\bar{\frakm}} \nc{\frakn}{\mathfrak n} \nc{\ddeg}[1]{^{(#1)}} \nc{\opset}{X} \nc{\genset}{{Z}} \nc{\NC}{\mathrm{{NC}}} \nc{\leaf}{\mathrm{leaf}} \nc{\twig}{\mathrm{twig}} \nc{\fe}{\mathrm{fl}} \nc{\munderline}[1]{#1} \nc{\bo}{o} \nc{\dep}{\mathrm{depth}} \nc{\ofe}{\mathrm{ofl}} \nc{\dfe}{\mathrm{dfe}} \nc{\fex}{\mathrm{fex}} \nc{\dl}{\mathrm{dlex}} \nc{\db}{\mathrm{db}} \nc{\lex}{\mathrm{lex}} \nc{\clex}{\mathrm{clex}} \nc{\dgp}{\mathrm{dgp}} \nc{\dgx}{\mathrm{dgx}} \nc{\br}{\mathrm{br}} \nc{\obd}{\mathrm{odb}} \nc{\ob}{\mathrm{ob}}

\nc{\bin}[2]{ (_{\stackrel{\scs{#1}}{\scs{#2}}})}  
\nc{\binc}[2]{ \left (\!\! \begin{array}{c} \scs{#1}\\
    \scs{#2} \end{array}\!\! \right )}  
\nc{\bincc}[2]{  \left ( {\scs{#1} \atop
    \vspace{-1cm}\scs{#2}} \right )}  
\nc{\bs}{\bar{S}} \nc{\cosum}{\sqsubset} \nc{\la}{\longrightarrow} \nc{\rar}{\rightarrow} \nc{\dar}{\downarrow} \nc{\dprod}{**} \nc{\dap}[1]{\downarrow \rlap{$\scriptstyle{#1}$}} \nc{\md}{\mathrm{dth}} \nc{\uap}[1]{\uparrow \rlap{$\scriptstyle{#1}$}} \nc{\defeq}{\stackrel{\rm def}{=}} \nc{\disp}[1]{\displaystyle{#1}} \nc{\dotcup}{\ \displaystyle{\bigcup^\bullet}\ } \nc{\gzeta}{\bar{\zeta}} \nc{\hcm}{\ \hat{,}\ } \nc{\hts}{\hat{\otimes}} \nc{\barot}{{\otimes}} \nc{\free}[1]{\bar{#1}} \nc{\uni}[1]{\tilde{#1}} \nc{\hcirc}{\hat{\circ}} \nc{\leng}{\ell} \nc{\lleft}{[} \nc{\lright}{]} \nc{\lc}{\lfloor} \nc{\rc}{\rfloor}
\nc{\lb}{[} 
\nc{\rb}{]} 
\nc{\curlyl}{\left \{ \begin{array}{c} {} \\ {} \end{array}
    \right.  \!\!\!\!\!\!\!}
\nc{\curlyr}{ \!\!\!\!\!\!\!
    \left. \begin{array}{c} {} \\ {} \end{array}
    \right \} }
\nc{\longmid}{\left | \begin{array}{c} {} \\ {} \end{array}
    \right. \!\!\!\!\!\!\!}
\nc{\onetree}{\bullet} \nc{\ora}[1]{\stackrel{#1}{\rar}}
\nc{\ola}[1]{\stackrel{#1}{\la}}
\nc{\ot}{\otimes} \nc{\mot}{{{\boxtimes\,}}} \nc{\otm}{\overline{\boxtimes}} \nc{\sprod}{\bullet} \nc{\scs}[1]{\scriptstyle{#1}} \nc{\mrm}[1]{{\rm #1}} \nc{\msum}{\sum\limits}
\nc{\margin}[1]{\marginpar{\rm #1}}   
\nc{\dirlim}{\displaystyle{\lim_{\longrightarrow}}\,} \nc{\invlim}{\displaystyle{\lim_{\longleftarrow}}\,} \nc{\mvp}{\vspace{0.3cm}} \nc{\tk}{^{(k)}} \nc{\tp}{^\prime} \nc{\ttp}{^{\prime\prime}} \nc{\svp}{\vspace{2cm}} \nc{\vp}{\vspace{8cm}} \nc{\proofbegin}{\noindent{\bf Proof: }}
\nc{\proofend}{$\blacksquare$ \vspace{0.3cm}}
\nc{\modg}[1]{\!<\!\!{#1}\!\!>}
\nc{\intg}[1]{F_C(#1)} \nc{\lmodg}{\!<\!\!} \nc{\rmodg}{\!\!>\!} \nc{\cpi}{\widehat{\Pi}}
\nc{\sha}{{\mbox{\cyr X}}}  
\nc{\shap}{{\mbox{\cyrs X}}} 
\nc{\shpr}{\diamond}    
\nc{\shp}{\ast} \nc{\shplus}{\shpr^+}
\nc{\shprc}{\shpr_c}    
\nc{\msh}{\ast} \nc{\zprod}{m_0} \nc{\oprod}{m_1} \nc{\vep}{\varepsilon} \nc{\labs}{\mid\!} \nc{\rabs}{\!\mid}

\nc{\dth}{d} \nc{\mmbox}[1]{\mbox{\ #1\ }} \nc{\fp}{\mrm{FP}} \nc{\rchar}{\mrm{char}} \nc{\Fil}{\mrm{Fil}} \nc{\Mor}{Mor\xspace} \nc{\gmzvs}{gMZV\xspace} \nc{\gmzv}{gMZV\xspace} \nc{\mzv}{MZV\xspace} \nc{\mzvs}{MZVs\xspace} \nc{\Hom}{\mrm{Hom}} \nc{\id}{\mrm{id}} \nc{\im}{\mrm{im}} \nc{\incl}{\mrm{incl}} \nc{\map}{\mrm{Map}} \nc{\mchar}{\rm char} \nc{\nz}{\rm NZ} \nc{\supp}{\mathrm Supp}

\nc{\Alg}{\mathbf{Alg}} \nc{\Bax}{\mathbf{Bax}} \nc{\bff}{\mathbf f} \nc{\bfk}{{\bf k}} \nc{\bfone}{{\bf 1}} \nc{\bfx}{\mathbf x} \nc{\bfy}{\mathbf y}
\nc{\base}[1]{\bfone^{\otimes ({#1}+1)}} 
\nc{\Cat}{\mathbf{Cat}} \delete{}
\nc{\detail}{\marginpar{\bf More detail}
    \noindent{\bf Need more detail!}
    \svp}
\nc{\Int}{\mathbf{Int}} \nc{\Mon}{\mathbf{Mon}}
\nc{\rbtm}{{shuffle }} \nc{\rbto}{{Rota-Baxter }} \nc{\remarks}{\noindent{\bf Remarks: }} \nc{\Rings}{\mathbf{Rings}} \nc{\Sets}{\mathbf{Sets}}

\nc{\BA}{{\Bbb A}} \nc{\CC}{{\Bbb C}} \nc{\DD}{{\Bbb D}} \nc{\EE}{{\Bbb E}} \nc{\FF}{{\Bbb F}} \nc{\GG}{{\Bbb G}} \nc{\HH}{{\Bbb H}} \nc{\LL}{{\Bbb L}} \nc{\NN}{{\Bbb N}} \nc{\KK}{{\Bbb K}} \nc{\QQ}{{\Bbb Q}} \nc{\RR}{{\Bbb R}} \nc{\TT}{{\Bbb T}} \nc{\VV}{{\Bbb V}} \nc{\ZZ}{{\Bbb Z}}

\nc{\cala}{{\mathcal A}} \nc{\calc}{{\mathcal C}} \nc{\cald}{{\mathcal D}} \nc{\cale}{{\mathcal E}} \nc{\calf}{{\mathcal F}} \nc{\calg}{{\mathcal G}} \nc{\calh}{{\mathcal H}} \nc{\cali}{{\mathcal I}} \nc{\call}{{\mathcal L}} \nc{\calm}{{\mathcal M}} \nc{\caln}{{\mathcal N}} \nc{\calo}{{\mathcal O}} \nc{\calp}{{\mathcal P}} \nc{\calr}{{\mathcal R}} \nc{\cals}{{\mathcal S}} \nc{\calt}{{\mathcal T}} \nc{\calw}{{\mathcal W}} \nc{\calk}{{\mathcal K}} \nc{\calx}{{\mathcal X}} \nc{\CA}{\mathcal{A}}

\nc{\fraka}{{\mathfrak a}} \nc{\frakA}{{\mathfrak A}} \nc{\frakb}{{\mathfrak b}} \nc{\frakB}{{\mathfrak B}} \nc{\frakD}{{\mathfrak D}} \nc{\frakH}{{\mathfrak H}} \nc{\frakM}{{\mathfrak M}} \nc{\bfrakM}{\overline{\frakM}} \nc{\frakm}{{\mathfrak m}} \nc{\frakP}{{\mathfrak P}} \nc{\frakN}{{\mathfrak N}} \nc{\frakp}{{\mathfrak p}} \nc{\frakS}{{\mathfrak S}} \nc{\frakx}{{\mathfrak x}} \nc{\ox}{\bar{\frakx}} \nc{\frakX}{{\mathfrak X}} \nc{\fraky}{{\mathfrak y}}
\nc{\frakW}{{\mathfrak W}}
\nc\dop{\delta}

\font\cyr=wncyr10 \font\cyrs=wncyr7

\begin{document}
\title{Relative locations of subwords in free operated semigroups and Motzkin words}

\author{Li Guo}
\address{
    Department of Mathematics and Computer Science,
         Rutgers University,
         Newark, NJ 07102, USA}
\email{liguo@rutgers.edu}

\author{Shanghua Zheng}
\address{Department of Mathematics,
    Lanzhou University,
    Lanzhou, Gansu 730000, China}
\email{zheng2712801@163.com}

\hyphenpenalty=8000
\date{\today}

\begin{abstract}
Bracketed words are basic structures both in mathematics (such as Rota-Baxter algebras) and mathematical physics (such as rooted trees) where the locations of the substructures are important.
In this paper we give the classification of the relative locations of two bracketed subwords of a bracketed word in an operated semigroup into the separated, nested and intersecting cases. We achieve this by establishing a correspondence between relative locations of bracketed words and those of words by applying the concept of Motzkin words which are the algebraic forms of Motzkin paths.
\end{abstract}

\subjclass[2010]{20M05, 20M99, 05E15, 16S15, 08B20}

\keywords{bracketed word, relative locations, operated semigroup, Motzkin word, Motskin path, rooted tree\\
Corresponding author: Li Guo, Department of Mathematics and Computer Science,
         Rutgers University,
         Newark, NJ 07102, USA; E-mail: liguo@rutgers.edu; Phone: 973-353-3917; Fax: 973-353-5270}

\maketitle

\tableofcontents

\hyphenpenalty=8000 \setcounter{section}{0}
\section{Introduction}
As a basic property of sets, there are three relative locations of any two subsets of a given set: separated (disjoint), nested (including) and intersecting (overlapping). See the proof of Theorem~\mref{thm:relw} for example. Similarly, there are three relative locations of two subwords in a given word, a property that is essential in rewriting systems (critical pairs) and Gr\"obner bases~\mcite{BN}. Analogous classification of relative locations of combinatorial objects, such as Feynman graphs, plays an important role in combinatorics and physics, for example in the renormalization of quantum field theory~\mcite{BW,CK2,Kr1,KW}.

The classification of relative locations can be quite subtle in some structures, especially when a non-identity operator is present, such as in differential algebras, Rota-Baxter algebras~\mcite{Gub} and, more generally, operated algebras~\mcite{BCD,BCQ,GSZ,ZGS}.
Further by~\mcite{Gop} free operated semigroups have natural combinatorial presentation as rooted trees which serve as the baby models for Feynman graphs~\mcite{CK,Kr1}.
Considering the importance of such classification in the study of these mathematics and physics structures, especially their Gr\"obner-Shirshov bases (compositions and diamond lemma)~\mcite{BCC,BCL}, it is crucial to establish such a classification. In this paper, we give an explicit correspondence between relative locations of two bracketed subwords and those of Motzkin words and subwords. As a direct consequence, we obtain the classification of the relative locations of bracketed subwords.

We take two steps in our treatment to deal with two subtle points of bracketed words. In Section~\mref{sec:relw} we deal with the first subtle point which is already present in studying relative locations of two subwords in comparison with two subsets, namely that one subword can appear at multiple locations in a given word. The concept of a $\star$-word~\mcite{BCQ} allows us to give a unique label to each appearance of a subword, called a \pla. We then show that each \pla corresponds uniquely to a substring of the string corresponding to the given word, converting the problem of studying subwords to that of subsets which can be solved easily as mentioned above. The second subtle point arise when we deal with bracketed words in Section~\mref{sec:bra} since the action of the bracket together with the product of the word gives the bracketed words a quite complicated structure. To resolve this difficulty, we make use of a bijection introduced in~\mcite{Gop} between bracketed words and a class of words called Motzkin words on a larger set. We then show in Section~\mref{sec:relb} that this bijection preserves the relative locations of the (bracketed) word, reducing the study of relative locations of bracketed words to the one of words for which we can apply Section~\mref{sec:relw}.

\section{Relative locations of subwords}
\mlabel{sec:relw}
In this section, we consider the relative locations of two subwords of a fixed word. This serves as both the prototype and preparation for our study of the relative locations of two bracketed subwords in later sections.

\subsection{Subwords}
\mlabel{ss:subw}

\begin{defn}
Let $Z$ be a set. Let $M(Z)$ be the free monoid on $Z$ consisting of {\bf words} in the alphabet set $Z$. Thus a word is either the {\bf empty word $1$} or of the form $w=z_1\cdots z_n, z_i\in Z$, $1\leq i\leq n.$ A {\bf subword} of $w$ is defined to be a word $u\neq 1$ that is a part of $w$. Let $S(Z):=M(Z)\backslash \{1\}$ be the free semigroup on $Z$.
\end{defn}

We emphasize that the empty word $1$ is not taken to be a subword in this paper.

Note that a subword $u$ of $w\in M(Z)$ may appear in $w$ at multiple locations. For example, for $z\in Z$, $u:=z$ appears in  $w:=zzz$ at three different locations. It is often important to be precise about the location of a subword. This is the case, for example, in the study of rewriting systems and Gr\"obner-Shirshov bases~\mcite{BN,BCQ,GSZ,ZGS}. For this purpose, we need additional information for the location of the subword $u$. Such information is provided
by the concept of $\star$-words~\mcite{BCQ}.
\begin{defn}
Let $Z$ be a set.
\begin{enumerate}
\item Let $\star$ be an element not in $Z$.  A {\bf $\star$-word on $Z$} is a word in $M(Z\cup\{\star\})$ in which $\star$ appears exactly once. The set of $\star$-words on $Z$ is denoted by $M^{\star}(Z)$.
\item
More generally, for $\star_1,\cdots,\star_k$ not in $Z$, denote $\vec{\star}=(\star_1,\cdots,\star_k)$. A {\bf$\vec{\star}$-word on $Z$} is a word in $M(Z\cup\{\star_1,\cdots,\star_k\})$ in which $\star_i$ appears exactly once for each $i=1,\cdots,k$. The set of $\vec{\star}$-words is denoted by  $M^{\vec{\star}}(Z)$. When $k=2$, it is denoted by $M^{\star_1,\star_2}(Z)$.
\item
For $p\in M^{\vec{\star}}(Z)$ and $u_1,\cdots, u_k\in M(Z)$, let $p|_{u_1,\cdots,u_k}=p|_{\star_1\mapsto u_1,\cdots,\star_k\mapsto u_k}$ denote the word in $M(Z)$ when the $\star_i$ in $p$ is replaced by $u_i$, where $i=1,\cdots,k$.
\end{enumerate}
\end{defn}

Now we can be more precise on a particular appearance of a subword.

\begin{defn}
Let $u, w\in M(Z)$ be words on $Z$. A {\bf \pla} of $u$ in $w$ (by $p$) is a pair $(u,p)$ where $p$ is in $M^{\star}(Z)$ such that $p|_u=w$.
\mlabel{defn:occst}
\end{defn}
Thus $(u_1,p_1)=(u_2,p_2)$ means $u_1=u_2$ and $p_1=p_2$.

Of course $u$ is a subword of $w$ if there is $p\in M^\star(Z)$ such that $(u,p)$ is a \pla of $u$ in $w$. However, the usefulness of the \pla notion is its role in distinguishing different appearances of $u$ in $w$. For example, the three appearances of $u=z$ in $w=zzz$ are identified by the three \plas $(u,p_1)$, $(u,p_2)$ and $(u,p_3)$ where $p_1=\star zz$, $p_2=z\star z$ and $p_3=zz\star$.

The concept of a \pla is also essential in determining the relative locations of two subwords of a given word.

\begin{exam}
Let $Z=\{x,y\}$ and let $w=xyxyxy$. Then $u:=xyx$ appears at two locations in $w$ and $v:=xy$ appears at three locations, as shown in the following equation.
$$
w=\underbrace{xyx}_{(u,\,p)}yxy
=\overbrace{xy}^{(v,\,q_1)}
\overbrace{xy}^{(v,\,q_2)}\overbrace{xy}^{
(v,\,q_3)}.
$$
Take $(u,p)$ to be the first appearance (from the left) of $u$ in $w$. Thus $p=\star yxy$. Then the three \plas of $v=xy$ in $w$, given by $(v,q_i), i=1,2,3,$ with $q_1=\star xyxy$, $q_2=xy\star xy$ and $q_3=xyxy\star$). These three placements of $v$ are in three different kinds of relative locations with respect to $u$: the left $v$ (in $(v,q_1)$) is a subword of $u$, the middle $v$ (in $(v,q_2)$) is not a subword of $u$ but has a nonempty intersection with $u$, and the right $v$ (in $(v,,q_3)$) is disjoint with $u$.
\mlabel{ex:1-3}
\end{exam}
This situation can again be made precise by $\star$-words.

\begin{defn}
Let $w$ be a word in $M(Z)$. Two \plas  $(u_1,p_1)$ and  $(u_2,p_2)$ are called
\begin{enumerate}
\item
{\bf separated} if
there exists an element $p $ in $ M^{\star_1,\star_2}(Z)$ such that $w=p|_{u_1,u_2}$, $p_1|_{\star_1}=p|_{\star_1,u_2}$ and $p_2|_{\star_2}= p|_{u_1,\star_2}$;
\mlabel{it:smm1}
\item
{\bf nested} if there exists an element $p$ in $M^{\star}(Z)$ such that  $p_1=p_2|_p$, or
 $p_2=p_1|_p$;
\mlabel{it:smm2}
\item
{\bf intersecting} if there exist an element
$p$ in $M^{\star}(Z)$ and elements $a,b,c$ in $ S(Z)$ such that
\begin{enumerate}
\item
either $w=p|_{abc}, p_1=p|_{\star c}, p_2=p|_{a\star}$;
\item
or $w=p|_{abc}, p_1=p|_{a\star}, p_2=p|_{\star c}$.
\end{enumerate}
\mlabel{it:smm3}
\end{enumerate}
\mlabel{defn:smm}
\end{defn}

\begin{exam}
With the $w, u$ and the three appearances of $v$ in Example~\mref{ex:1-3}, we have the corresponding \plas $(u,p)$ with $p=\star yxy$, and $(v,q_1)$, $(v,q_2)$ and $(v,q_3)$. Then $(u,p)$ and $(v,q_1)$ (resp. $(v,q_2)$, resp. $(v,q_3)$) are nested (resp. intersecting, resp. separated).
\mlabel{ex:1-3b}
\end{exam}

\subsection{Substrings}
\mlabel{ss:subst}
We now give another description of a \pla of a subword that makes it easier to classify the relative locations of two subwords.
Denote $[n]:=\{1,\cdots,n\}$ and $[i,k]:=\{i,\cdots,k\}$ for $n, i\geq 1, k\geq i$.

Let $Z$ be a set. Let $w=z_1\cdots z_n, z_i\in Z$, $1\leq i\leq n,$ be a word in $M(Z)$. Let $\pl:=(u,p)$ be a \pla of $u$ in $w$. Then $p$ is of the form $p=x_1\cdots x_j \star x_{j+1}\cdots x_m$ with $x_i\in Z, 1\leq i\leq m,$ and $u=y_1\cdots y_t$ with $y_j\in Z, 1\leq j\leq t$. Comparing
$$ x_1\cdots x_j y_1\cdots y_t x_{j+1}\cdots x_m=p|_u=w =z_1\cdots z_n$$
in the free monoid $M(Z)$, we obtain $p=z_1\cdots z_{j-1}\star z_{k+1}\cdots z_n$ for unique $j=j_\pl$ and $k=k_\pl$. Thus $u=z_j\cdots z_k$.

\begin{defn}
With notations as above, the set $I_\pl:=[j_\pl,k_\pl]$ is called the {\bf \loc} of the \pla $\pl=(u,p)$ in $w$. Denote \begin{equation}
\mathcal{LO}(w):=\{[j,k]\,|\,  1\leq j\leq k\leq n\}.
\mlabel{eq:lo}
\end{equation}
\end{defn}
Also denote
\begin{equation}
\mathcal{PL}(w):=\{(u,p)\,|\, (u,p) \text{ is a \pla in } w\}.
\mlabel{eq:pla}
\end{equation}

\begin{prop}
Let $1\neq w\in M(Z)$. The map
$$\etaup: \mathcal{PL}(w)\longrightarrow  \mathcal{LO}(w),\quad
\pl=(u,p)\mapsto [j_\pl,k_\pl], $$
is a bijection.
\end{prop}

\begin{proof}
Let $w=z_1\cdots z_n$. Given a \pla $(u,p)$ in $w$. Then $p=z_1\cdots z_{j-1}\star z_{k+1}\cdots z_{n}\in M^{\star}(Z)$ and $u=z_j\cdots z_k, 1\leq j\leq k\leq n$. Thus if $\pl_1=(u_1,p_1)$ and $\pl_2=(u_2,p_2)$ are in $\mathcal{PL}(w)$ and $\pl_1\neq \pl_2$, then we have $p_1\neq p_2$. Hence either $j_{\pl_1}\neq j_{\pl_2}$ or $k_{\pl_1}\neq k_{\pl_2}$. Then $\etaup(u_1,p_1)\neq \etaup(u_2,p_2)$. Hence $\etaup$ is injective.

Further for any $[j,k]\in \mathcal{LO}(w)$, define $u=z_j\cdots z_k$ and $p=z_1\cdots z_{j-1}\star z_{k+1}\cdots z_n$. Then we have $p|_u=w$ and $\etaup(u,p)=[j,k]$. Hence $\etaup$ is surjective.
\end{proof}

\begin{defn}
Two nonempty subsets $I$ and $J$ of $[n]$ are called
\begin{enumerate}
\item {\bf separated} if $I\cap J =\emptyset$;
\item {\bf nested} if $I\subseteq J$ or $J\subseteq I$;
\item {\bf intersecting} if $I\cap J \neq \emptyset$, $I\not \subseteq J$ and $J\not \subseteq I$.
\end{enumerate}
\end{defn}

Consider the $w, u$ and $v$ in Example~\mref{ex:1-3} and their corresponding \plas $(u,p)$, $(v,q_1)$, $(v,q_2)$ and $(v,q_3)$ in Example~\mref{ex:1-3b}. The \locs of the \plas are $\etaup(u,p)=[1,3]$ and $\etaup(v,q_1)=[1,2], \etaup(v,q_2)=[3,4]$ and $\etaup(v,q_3)=[5,6]$. Then we see that, as subsets of $[6]$, $[1,3]$ and $[1,2]$ (resp. $[3,4]$, resp. $[5,6]$) are also nested (resp. intersecting, resp. separated).
The next theorem shows that this equivalence holds in general.

\begin{theorem}
Let $Z$ be a set. Let $1\neq w=z_1\cdots z_n$ be in $M(Z)$ where $z_i\in Z, 1\leq i\leq n$. Then  \plas $\pl_1=(u_1,p_1)$ and $\pl_2=(u_2,p_2)$ in $w$ are separated (resp. nested, resp. intersecting) if and only if the corresponding subsets $I_{\pl_1}$ and $I_{\pl_2}$ of $[n]$ are separated (resp. nested, resp. intersecting).
\mlabel{thm:main3}
\end{theorem}
\begin{proof}
First suppose that the  \plas $(u_1,p_1)$ and $(u_2,p_2)$ in $w$  are separated. Then there exists an element $p$ in $M^{\star_1,\star_2}(Z)$ such that
$$w=p|_{u_1,u_2}, \quad p_1|_{\star_1}=p|_{\star_1,u_2}, \quad p_2|_{\star_2}=p|_{u_1,\star_2}.$$
From $p|_{u_1,u_2}=z_1\cdots z_n$, there are $1\leq j_1\leq k_1<j_2\leq k_2\leq n$  such that
$$p=z_1\cdots z_{j_1-1}\star_1 z_{k_1+1}\cdots z_{j_2-1}\star_2 z_{k_2+1}\cdots z_n,$$
or
$$p=z_1\cdots z_{j_1-1}\star_2 z_{k_1+1}\cdots z_{j_2-1}\star_1 z_{k_2+1}\cdots z_n.$$
Thus we have $\{I_{\pl_1}, I_{\pl_2}\}=\{[j_1,k_1], [j_2,k_2]\}$ and hence $I_{\pl_1}\cap I_{\pl_2}=\emptyset$.
Conversely, suppose $I_{\pl_1}\cap I_{\pl_2}=\emptyset$. Let $I_{\pl_1}=[j_1,k_1]$ and $I_{\pl_2}=[j_2,k_2]$. Then we have $k_1<j_2$ or $k_2<j_1$. Without loss of generality, assume $k_1<j_2$. Then take $u_1=z_{j_1}\cdots z_{k_1}$ and $u_2=z_{j_2}\cdots z_{k_2}$. Define $p=z_1\cdots z_{j_1-1} \star_1 z_{k_1+1}\cdots z_{j_2-1} \star_2 z_{k_2+1}\cdots z_n$. Then we have $$w=p|_{u_1,u_2}, \quad p_1|_{\star_1}=p|_{\star_1,u_2},\quad  p_2|_{\star_2}=p|_{u_1,\star_2},$$
as needed.
\smallskip

Next suppose that $(u_1,p_1)$ and $(u_2,p_2)$ are nested. Then there exists an element $p$ in $M^{\star}(Z)$ such that $p_1=p_2|_{p}$ or $p_2=p_1|_{p}$. Without loss of generality, we can assume $p_1=p_2|_p$. From $w=p_1|_{u_1}$ and $w=p_2|_{u_2}$ we obtain
$$z_1\cdots z_{j_1-1}\star z_{k_1+1}\cdots z_n=p_1=
p_2|_p=z_1\cdots z_{j_2-1}\star|_p z_{k_2+1}\cdots z_n.$$
Thus $j_2\leq j_1\leq k_1\leq k_2$ and hence $I_{\pl_1}\subseteq I_{\pl_2}$.
Conversely, suppose $I_{\pl_1}$ and $I_{\pl_2}$ are nested. We may assume that $I_{\pl_1}=[j_1,k_1]\subseteq I_{\pl_2}=[j_2,k_2]$. Then define $p=z_{j_2}\cdots z_{j_1-1}\star z_{k_1+1}\cdots z_{k_2}$, with the convention that $p=\star z_{k_1+1}\cdots z_{k_2}$ if $j_1=j_2$ and $p=z_{j_2}\cdots z_{j_1-1}\star$ if $k_1=k_2$. We have $p_2|_p=p_1$. Hence $(u_1,p_1)$ and $(u_2,p_2)$ are nested. \smallskip

Finally, suppose that $(u_1,p_1)$ and $(u_2,p_2)$ are intersecting. Then there exist $p$ in $M^{\star}(Z)$ and $a,b,c$ in $ S(Z)$ such that
$$w=p|_{abc}, \quad p_1=p|_{\star c}, \quad p_2=p|_{a\star}, $$
or
$$w=p|_{abc}, \quad p_1=p|_{a\star}, \quad p_2=p|_{\star c}.$$
We just need to consider the first case since the proof of the second case is similar. Denote $p=z_1\cdots z_{j-1}\star z_{k+1}\cdots z_n$. From
$$z_1\cdots z_{j_1-1}\star z_{k_1+1}\cdots z_n=p_1
=p|_{\star c}= z_1\cdots z_{j-1}\star c z_{k+1}\cdots z_n$$
we obtain $j_1=j$ and $c=z_{k_1+1}\cdots z_k$. Since $c \neq 1$, we have $k_1<k$.
Similarly, from
$$z_1\cdots z_{j_2-1}\star z_{k_2+1}\cdots z_n=p_2
=p|_{a\star}= z_1\cdots z_{j-1}a\star z_{k+1}\cdots z_n$$
we obtain $k_2=k$ and $a=z_{j}\cdots z_{j_2-1}$. Since $a \neq 1$, we have $j<j_2$. Consequently, 
$$ w=p|_{abc}=z_1\cdots z_{j-1}abcz_{k+1}\cdots z_n=z_1\cdots z_{j_2-1}b z_{k_1+1}\cdots z_n.$$
Since $b\neq 1$, we have $j_2\leq k_1$. Then $$j_1=j<j_2\leq k_1<k=k_2.$$
Thus
$$I_{\pl_1}\cap I_{\pl_2}=[j_2,k_1]\neq \emptyset, \quad I_{\pl_1}\backslash I_{\pl_2}=[j_1,j_2-1]\neq \emptyset, \quad I_{\pl_2}\backslash I_{\pl_1}=[k_1+1,k_2]\neq \emptyset.$$
Hence $I_{\pl_1}\nsubseteq I_{\pl_2}$ and $I_{\pl_2}\nsubseteq I_{\pl_1}$. Thus $I_{\pl_1}$ and $I_{\pl_2}$ are intersecting.

Conversely, suppose that $I_{\pl_1}\cap I_{\pl_2} \neq \emptyset$, $I_{\pl_1}\not \subseteq I_{\pl_2}$ and $I_{\pl_2}\not \subseteq I_{\pl_1}$. From $I_{\pl_1}\cap I_{\pl_2} \neq \emptyset$, we have $j_1\leq k_2$ and $j_2\leq k_1$. Then from $I_{\pl_1}\not \subseteq I_{\pl_2}$ and $I_{\pl_2}\not \subseteq I_{\pl_1}$, we have $k_1>k_2$ and $j_1>j_2$, or $k_1<k_2$ and $j_1<j_2$. We just consider the first case with the second case being similar. Then we have
$j_2<j_1\leq k_2<k_1$. Define
$$p=z_1\cdots z_{j_2-1}\star z_{k_1+1}\cdots z_n, \quad a=z_{j_2}\cdots z_{j_1-1}, \quad b=z_{j_1}\cdots z_{k_2}, \quad c=z_{k_2+1}\cdots z_{k_1}.$$
Then we have
$$ a, b, c\neq 1, \quad w=p|_{abc},\quad p_2=p|_{\star c},\quad p_1=p|_{a\star}.$$
This shows that $(u_1,p_1)$ and $(u_2,p_2)$ are intersecting.
\end{proof}

The following result is the classification of the relative locations of two \plas in the free monoid $M(Z)$. Its generalization to operated monoids will be treated in the subsequent sections.

\begin{theorem}
Let $1\neq w$ be a word in $M(Z)$. For any two \plas $\pl_1=(u_1, p_1)$ and $\pl_2=(u_2,p_2)$ in $w$,  exactly one of the following is true:
\begin{enumerate}
\item
$(u_1, p_1)$ and $(u_2,p_2)$ are separated;
\item
$(u_1, p_1)$ and $(u_2,p_2)$ are nested;
\item
$(u_1, p_1)$ and $(u_2,p_2)$ are intersecting.
\end{enumerate}
\mlabel{thm:relw}
\end{theorem}
\begin{proof}
Let $w=z_1\cdots z_n$ with $z_1,\cdots,z_n \in Z$, $n\geq 1$. By Theorem~\mref{thm:main3}, we only need to prove that the same conclusion holds for $I_{\pl_1}$ and $I_{\pl_2}$. But this follows from the simple fact that, for the two subsets $I_{\pl_1}$ and $I_{\pl_2}$ of $[n]$, exactly one of the following is true:
\begin{enumerate}
\item $I_{\pl_1}$ and $I_{\pl_2}$ have empty intersection ($\Leftrightarrow$ $I_{\pl_1}$ and $I_{\pl_2}$ are separated);
\item $I_{\pl_1}$ and $I_{\pl_2}$ have nonempty intersection, and $I_{\pl_1}\subseteq I_{\pl_2}$ or $I_{\pl_2}\subseteq I_{\pl_1}$ ($\Leftrightarrow$ $I_{\pl_1}$ and $I_{\pl_2}$ are nested);
\item or $I_{\pl_1}$ and $I_{\pl_2}$ have nonempty intersection, and $I_{\pl_1}\not\subseteq I_{\pl_2}$ and $I_{\pl_2}\not\subseteq I_{\pl_1}$ ($\Leftrightarrow$ $I_{\pl_1}$ and $I_{\pl_2}$ are intersecting).
\end{enumerate}
\end{proof}

\section{Bracketed words and Motzkin words}
\mlabel{sec:bra}
We recall the concepts of bracketed words and Motzkin words \mcite{Al,Gop,ST,GSZ} before generalizing Theorem~\mref{thm:relw} to this context.

\subsection{Bracketed words}
We first recall the concept and construction of free operated monoids, following~\mcite{Gop,GSZ}.

\begin{defn}
{\rm An {\bf operated monoid}  is a monoid $U$ together with a map  $P: U\to U$. A {\bf morphism} from an operated monoid\,  $U$ with a map $P:U\to U$ to an operated monoid $V$ with a map $Q:V\to V$ is a monoid homomorphism $f :U\to V$ such that $f \circ P= Q \circ f$. } \mlabel{de:mapset}
\end{defn}
\begin{defn}
A {\bf free operated monoid } on a set $Y$ is an operated monoid $(U_Y,P_Y)$ together with a map $j_Y: Y\rightarrow U_Y$ with the property that, for any operated monoid  $(V,Q)$ and any map $f:Y\rightarrow V$, there is a unique morphism $\bar{f}: (U_Y,P_Y)\rightarrow (V,Q)$ of operated monoids  such that $f=\bar{f}\circ j_Y$.
\end{defn}

For any set $Y$,  let $\lc Y\rc:=\{\lc y\rc \,|\, y\in Y\}$ denote a set indexed by, but  disjoint from $Y$. Let $X$ be a given set.
We will construct a direct system $\{\frakM_n:=\frakM_n(X)\}_{n \geq 0}$ of monoids with natural embeddings $\uni{i}_{n-1}:\frakM_{n-1}\hookrightarrow
    \frakM_{n}$ for $n \geq 1$ by
induction on $n$. The free operated monoid $\frakM(X)$ on $X$ is the direct limit of the system, after we equip it with a natural operator.

First let $\frakM_0:=M(X)$. Then define the monoid
$$\frakM_1:=M(X\cup \lc \frakM_0 \rc)=M(X\cup \lc M(X)\rc)$$
and denote the natural embedding of monoids induced by the inclusion $X\hookrightarrow X\cup \lc \frakM_0\rc$ by
$$\uni{i}_{0}: \frakM_0=M(X)\hookrightarrow M(X\cup \lc
\frakM_0 \rc) = \frakM_1.$$

Inductively assuming that $\frakM_{n}$ and $\uni{i}_{n-1}:\frakM_{n-1}\hookrightarrow \frakM_{n}$ have been defined for $n\geq 1$, we define the monoid
\begin{equation}
 \frakM_{n+1}:=M(X\cup \lc\frakM_{n}\rc )
 \mlabel{eq:frakm}
 \end{equation}
From the embedding of monoids $\uni{i}_{n-1}: \frakM_{n-1} \hookrightarrow \frakM_{n}$, we have an embedding $ \lc\frakM_{n-1}\rc \hookrightarrow \lc \frakM_{n} \rc$. By the freeness of $\frakM_{n}=M(X\cup \lc\frakM_{n-1}\rc)$, we obtain a natural embedding of monoids
$$
\uni{i}_{n}:\frakM_{n} = M(X\cup \lc\frakM_{n-1}\rc)\hookrightarrow
    M(X\cup \lc \frakM_{n}\rc) =\frakM_{n+1}.
$$
This completes our inductive definition of the direct system. Let
$$ \frakM(X):=\bigcup_{n\geq 0}\frakM_n=\dirlim
\frakM_n$$ be the direct limit of the system. We note that $\frakM(X)$ is a monoid. By taking direct limit on both sides of $\frakM_n = M(X\cup \lc \frakM_{n-1}\rc)$, we obtain
\begin{equation}\frakM(X)=M(X\cup \lc \frakM(X)\rc)
\mlabel{eq:omid}\end{equation}
whose elements are called {\bf  bracketed words}.

The {\bf depth} of $f\in \frakM(X)$ is defined to be
\begin{equation}
\dep(f):=\min \left \{n\,\big|\, f\in \frakM_n\right\}.
\mlabel{eq:dep}
\end{equation}
The following result shows that $\frakM(X)$ is the equivalence of free monoids in the category of operated monoids.
\begin{lemma}{\bf \mcite{Gop}}
Define the map $\lc\ \rc: \frakM(X) \to \frakM(X)$ by taking $w \in \frakM(X)$ to $\lc w\rc$. Let $j_X:X \to \frakM(X)$ be the natural embedding. Then
the triple $(\frakM(X),\lc\ \rc, j_X)$ is the free operated monoid on $X$.
\mlabel{pp:freetm}
\end{lemma}
By \cite[Theorem~4.2]{Gop}, another representation of free operated semigroups is given by rooted trees. See~\mcite{CK,Kr1} for the application of rooted trees in quantum field theory.

\subsection{Motzkin words}
We now recall the definition of Motzkin words which aquired its name since it encodes Motzkin paths~\mcite{DS}. Motzkin words give another construction of free operated monoids~\mcite{Gop} and in this paper serve as the bridge between bracketed words and the usual words.

Let $X$ be a set. Let $\lm$ and $\rtm$ be symbols not in $X$.
\begin{defn}
An element of the free monoid $M(X\cup \{\lm,\rtm\})$ is called a {\bf Motzkin word} on $X$ if it has the properties that
\begin{enumerate}
\item
the number of $\lm$ in the word equals the number of $\rtm$ in the word;
\item
counting from the left, the number of occurrence of $\lm$ is always greater or equal to the number of occurrence of $\rtm$.
\end{enumerate}
The set of all Motzkin words is denoted by $\mathcal{W}(X)$. \mlabel{defn:mword}
\end{defn}

\begin{exam}
Let $x,y,z$ be elements in $X$.
\begin{enumerate}
\item The word $\lm \lm x\rtm \lm y\rtm \rtm z$ is a Motzkin word.
\item  the word $\lm x\rtm \lm yz$ is not a Motzkin word since it does not satisfy the first property.
\item The word $x\rtm\lm y\lm z\rtm$ is not a Motzkin  word since it does not satisfy the second property.
\end{enumerate}
\end{exam}

Define
$$P:\calw(X)\rightarrow \calw(X), \quad P(m)=\lm m\rtm, m\in \calw(X).$$
Then $\calw(X)$ is an operated monoid.

\begin{theorem} {\bf\cite[Theorem 3.4]{Gop}}
Let $Y$ be a set. Let
$$\phi:=\phi_Y:\frakM(Y)\rightarrow \mathcal{W}(Y), \quad \phi(u)=u, u\in Y,$$
be the homomorphism of operated monoid defined by the universal property of the free operated monoid $\frakM(Y)$.  Then $\phi$ is  an isomorphism. \mlabel{thm:mbiso}
\end{theorem}

\subsection{$\vs$-bracketed words and $\vs$-Motzkin words}
Let $X$ be a set. For $k\geq 1$,
let $\star_1,\cdots, \star_k$  be distinct symbols not in $X$ and let $\vs =(\star_1,\cdots,\star_k)$. Denote $X^{\vs}=X\cup\{\star_1,\cdots,\star_k\}$.

\begin{defn}
A {\bf $\vs$-bracketed word} on $X$ is defined to be a bracketed word in the operated monoid  $\frakM(X^{\vs})$ with exactly one occurrence of $\star_i$ for each $i=1,\cdots, k$.
The set of all $\vs$-bracketed words on $X$ is denoted by $\frakM^{\vs}(X)$. When $k=1$ and $2$, we denote
$\frakM^{\vs}(X)$ by $\frakM^\star(X)$ and $\frakM^{\star_1,\star_2}(X)$ respectively.
\end{defn}

\begin{defn}
For $q\in \frakM^{\vs}(X)$ and $s_1,\cdots,s_k\in \frakM(X)$, we define
\begin{equation}
q|_{\vec{s}}:=q|_{s_1,\cdots,s_k}
\mlabel{eq:2star1}
\end{equation}
to be the bracketed word in $\frakM(X)$ obtained by replacing the letter $\star_i$ in $q$ by $s_i$ for $1\leq i\leq k$.  A {\bf $(s_1,\cdots, s_k)$-bracketed word} on $X$ is a bracketed word of the form Eq.\,(\mref{eq:2star1}) for some $q \in \frakM^{\vs}(X)$.
\end{defn}

We next introduce the concept of a $\vs$-Motzkin word on $X$.

\begin{defn}
A word in $M(X^{\vs}\cup\{\lm,\rtm\})$ is called a {\bf $\vs$-Motzkin word} on $X$ if it is in the intersection $M^{\vs}(X\cup\{\lm,\,\rtm\})\cap \calw(X^{\vs})$ or, more precisely, if it has the properties that
\begin{enumerate}
\item
the number of $\lm$ in the word equals the number of $\rtm$ in the word;
\item
counting from the left, the number of occurrence of $\lm $ is always greater or equal to the number of occurrence of $\rtm$;
\item
for each $1\leq i\leq k$, the letter $\star_i$ appears exactly once in the word.
\end{enumerate}
The set of all  $\vs$-Motzkin words is denoted by $\mathcal{W}^{\vs}(X)$.\mlabel{defn:smword}
\end{defn}

Taking $Y=X^{\vs}$ in Theorem~\mref{thm:mbiso}, we obtain
\begin{coro}
There is a unique isomorphism
$$\phi_{X^{\vs}}:\frakM(X^{\vs})
\rightarrow \mathcal{W}(X^{\vs})$$
of operated monoids such that $\phi_{X^{\vs}}(a)=a$ for all $a\in X^{\vs}$. In particular, $\phi_{X^{\vs}}(\star_i)=\star_i$ for $1\leq i\leq k$.
\mlabel{co:sbm}
\end{coro}
\begin{prop}
The isomorphism $\phi_{X^{\vs}}$ restricts to a bijection $$\phi_{\vs}:\frakM^{\vs}(X) \rightarrow
\mathcal{W}^{\vs}(X).$$
\mlabel{prop:sbij}
\end{prop}
\begin{proof}
Since the bijection $\phi_{X^{\vs}}$ sends $\star_i$ to $\star_i, 1\leq i\leq k$, an element $w\in \frakM(X^{\vs})$ is in
$\frakM^{\vs}(X)$ if and only if $\phi_{X^{\vs}}(w)\in \calw(X^{\vs})$ is in $\calw^{\vs}(X)$, hence the proposition.
\end{proof}

\section{Relative locations in bracketed words and Motzkin words}
\mlabel{sec:relb}
In this section, we first define the relative locations of bracketed words and Motzkin words by using $\vs$-bracketed words and $\vs$-Motzkin words in Section~\mref{subsec:obra}. The relationship between the two kinds of relative locations is established in Section~\mref{subsec:place} and is used to prove the main theorem in Sections~\mref{subsec:pla}.

\subsection{\Plas  in bracketed words and Motzkin words}
\mlabel{subsec:obra}

\begin{defn}Let $X$ be a set. Let $f, s$ be in $\frakM(X)$. A {\bf \pla} of $s$ in $f$ (by $q$) is a pair $\pl:=(s,q)$ with $q\in\frakM^{\star}(X)$ such that $f=q|_s$. We then call $s$ a {\bf bracketed  subword of $f$}.
\mlabel{defn:brword}
\end{defn}

The relative locations between two \plas in a bracketed word can now be defined in the same way as those in an ordinary word.

\begin{defn}
Two \plas $(s_1,q_1)$ and $(s_2,q_2)$ in $f\in \frakM(X)$ are  called
\begin{enumerate}
\item
{\bf separated} if there exists an element $q$ in $\frakM^{\star_1,\star_2}(X)$ such that $q_1|_{\star_1}=q|_{\star_1,s_2}$, $q_2|_{\star_2}=q|_{s_1, \star_2}$ and $f=q|_{s_1,s_2}$;
\mlabel{item:bsep}
\item
{\bf nested} if there exists an element $q$ in $\frakM^{\star}(X)$ such that $q_1=q_2|_q$ or $q_2=q_1|_q$;
\mlabel{item:bnes}
\item
 {\bf intersecting} if there exist an element $q$ in $\frakM^{\star}(X)$ and elements $a,b,c$ in $\frakM(X)\backslash\{1\}$ such that
\begin{enumerate}
\item
either $f=q|_{abc}, q_1=q|_{\star c}, q_2=q|_{a\star}$;
\item
or $f=q|_{abc}, q_1=q|_{a\star}, q_2=q|_{\star c}$.
\end{enumerate}
\mlabel{item:bint}
\end{enumerate}
\mlabel{defn:bwrel}
\end{defn}

A bracketed word $s$ might appear in a bracketed word $f$ with multiple \plas with different relative locations with respect to another bracketed subword.

\begin{exam}
Let $f=\lc \lc abc\rc ab\rc\in \frakM(X)$. Let $s_1=\lc abc\rc$ and let $s_2=ab$. Then we obtain the \plas $(s_1,q_1)$, $(s_2, q_{21})$ and $(s_2,q_{22})$, where $q_1=\lc \star ab\rc$, $q_{21}=\lc \lc\star c\rc ab\rc$ and $q_{22}=\lc \lc abc\rc \star\rc$. Then the \plas $(s_1,q_1)$ and $(s_2,q_{21})$ are nested since $q_{21}=q_1|_{q}$, where $q=\lc \star c\rc$. However the \plas $(s_1,q_1)$ and $(s_2,q_{22})$ are separated.\mlabel{exam:ex1}
\end{exam}

We next describe the relative locations of two \plas of bracketed words in terms of Motzkin words.
By Definition~\mref{defn:mword}, we obtain
\begin{lemma}
Let $u$ be in $\calw(X)$ and $p$ be in $\calw^\star(X)$. Then $p|_u$ is in $\calw(X)$.
\mlabel{lem:msub}
\end{lemma}

Taking $Z=X\cup\{\lm,\rtm\}$ in Definition~\mref{defn:occst}, we have
\begin{defn}
Let $w$ and $u$ be Motzkin words in $\mathcal{W}(X)$. A {\bf \pla} of $u$ in $w$ (by $p$) is a pair $(u,p)$ with $p\in \calw^\star(X)$ such that $p|_u=w$. We then call $u$ a {\bf Motzkin subword} of $w$.
\mlabel{defn:occm}
\end{defn}

By Definition~\mref{defn:mword}, the set $\mathcal{W}^{\star_1,\star_2}(X)$ of $(\star_1,\star_2
)$-Motzkin words is a subset of  $M^{\star_1,\star_2}(X\cup\{\lm,\rtm\})$. Thus as a special case of Definition~\mref{defn:smm}, we obtain the definition of two \plas $(u_1,p_1)$ and $(u_2,p_2)$ in $w$ being {\bf separated} or {\bf nested} or {\bf intersecting}.

\subsection{Relationship between relative locations}
\mlabel{subsec:place}

We now establish the relationship between \plas in bracketed words and \plas in Motzkin words.
\begin{prop}
Let $X$ be a set and let $\{\star_1,\cdots,\star_k\}$ be symbols not in $X$.
Let $\phi:\frakM(X)\rightarrow \mathcal{W}(X)$ be the isomorphism of operated monoids in Theorem ~\mref{thm:mbiso}, and let $\phi_{\vs}: \frakM^{\vs}(X)\rightarrow \mathcal{W}^{\vs}(X)$ be the bijection in Proposition~\mref{prop:sbij}.
\begin{enumerate}
\item
Let $q\in \frakM^{\vs}(X)$ and $s_1,\cdots,s_k\in \frakM(X)$. Then
$\phi(q|_{s_1,\cdots,s_k}) =\phi_{\vs}(q)|_{\phi(s_1),\cdots,\phi(s_k)}$.
\mlabel{item:bmphi}
\item
Let $p\in \calw^{\vs}(X)$ and $u_1,\cdots,u_k\in \calw(X)$. Then $\phi^{-1}(p|_{u_1,\cdots,u_k})
=\phi_{\vs}^{-1}(p)|_{\phi^{-1}(u_1),\cdots,\phi^{-1}(u_k)}\,$.
\mlabel{item:mbphi}
\end{enumerate}
\mlabel{prop:bmrel}
\end{prop}

\begin{proof}
\noindent
(\mref{item:bmphi})
We prove a more general result.
Define
$$ \overline{\frakM}^{\,\vs}(X):=\left\{ q\in \frakM(X^{\vs})\,\big|\, \text{ for each } 1\leq i\leq k, \star_i \text{ appears at most once in } q\right\}\subseteq \frakM(X^{\vs}).$$
For $q\in \overline{\frakM}^{\,\vs}(X)$ and $s_1,\cdots,s_k\in \frakM(X)$, define $q|_{s_1,\cdots,s_k}$ to be the bracketed word by replacing the $\star_i$ in $q$, if there is any, by $s_i$, $1\leq i\leq k$. Thus in particular, $q|_{s_1,\cdots,s_k}=q|_{s_1}$ if $q\in \frakM^{\star_1}(X)$ and $q|_{s_1,\cdots,s_k}=q$ if $q\in \frakM(X)$. In general, $q|_{s_1,\cdots,s_k}=q|_{s_{i_1},\cdots,s_{i_m}}$ if only $\star_{i_1},\cdots,\star_{i_m}$ from $\{\star_1,\cdots,\star_k\}$ appear in $q$.

Similarly define
$$ \overline{\calw}^{\,\vs}(X):=\left\{ p\in \calw(X^{\vs})\,\big|\, \text{ for each } 1\leq i\leq k, \star_i \text{ appears at most once in } p\right\}\subseteq \calw(X^{\vs}),$$
and define $p|_{u_1,\cdots,u_k}$ for $p\in \overline{\calw}^{\,\vs}(X)$ and $u_1,\cdots,u_k\in \calw(X)$.

With these notations, if $q$ is in $\overline{\frakM}^{\,\vs}(X)$ and $q=q_1q_2$ with $q_1, q_2\in \overline{\frakM}^{\,\vs}(X)$, then we have
$q|_{\vec{s}}=q_1|_{\vec{s}}\,q_2|_{\vec{s}}.$
Further the bijection $\phi_{X^{\vs}}$ in Corollary~\mref{co:sbm} restricts to a bijection $$\overline{\phi}_{\vs}:\overline{\frakM}^{\,\vs}(X)\to \overline{\calw}^{\,\vs}(X)$$
by the same argument as Proposition~\mref{prop:sbij}.

\begin{claim}
Let $q\in \overline{\frakM}^{\,\vs}(X)$ and $s_1,\cdots,s_k\in \frakM(X)$. Then
$\phi(q|_{s_1,\cdots,s_k}) =\overline{\phi}_{\vs}(q)|_{\phi(s_1),\cdots,\phi(s_k)}$.
\mlabel{cl:bmphi}
\end{claim}
\begin{proof}
We prove the claim by induction on the depth of $q\in \overline{\frakM}^{\,\vs}(X)$ defined in Eq.~(\mref{eq:dep}).

If the depth of $q$ is 0, then $q\in \overline{M}^{\,\vs}(X)$ and hence
$q=z_1\star_{i_1}z_2\star_{i_2}\cdots z_m\star_{i_m} z_{m+1}$ with $z_1,\cdots, z_{m+1}\in
M(X)$ and $\star_{i_1},\cdots,\star_{i_m}\in \{\star_1,\cdots,\star_k\}$. Then we have
$$ \phi(q|_{\vec{s}}) =\phi(z_1s_{i_1}\cdots z_m s_{i_m} z_{m+1})
= \phi(z_1)\phi(s_{i_1})\cdots \phi(z_m)\phi(s_{i_m})\phi(z_{m+1}) =\overline{\phi}_{\vs}(q)|_{\phi(\vec{s})},$$
as needed. Here we have used the abbreviations $\vec{s}=(s_1,\cdots,s_k)$ and $\phi(\vec{s})=(\phi(s_1),\cdots,\phi(s_k))$.

Assume that the claim has been proved for $q\in \overline{\frakM}^{\,\vs}(X)$ with depth less or equal to $n\geq 0$ and consider a $q$ with depth $n+1$.

Note that $\phi$ and $\phi_{X^{\vs}}$ (and hence $\overline{\phi}_{\vs}$) are multiplicative. Thus for $q=q_1q_2$ with $q, q_1, q_2\in \overline{\frakM}^{\,\vs}(X)$, if we can prove that $\phi(q_i|_{\vec{s}}) =\overline{\phi}_{\vs}(q_i)|_{\phi(\vec{s})}, i=1,2$, then we also have
$$ \phi(q|_{\vec{s}}) =\phi(q_1|_{\vec{s}})\phi(q_2|_{\vec{s}}) =\overline{\phi}_{\vs}(q_1)|_{\phi(\vec{s})} \overline{\phi}_{\vs}(q_2)|_{\phi(\vec{s})} =(\overline{\phi}_{\vs}(q_1)\overline{\phi} _{\vs}(q_2))|_{\phi(\vec{s})}
=\overline{\phi}_{\vs}(q)|_{\phi(\vec{s})}.$$
Therefore, we only need to complete the induction when $q$ is of depth $n+1$ and is indecomposable in $\overline{\frakM}^{\,\vs}(X)$, that is, $q$ is not the product of two elements in $\overline{\frakM}^{\,\vs}(X)\subseteq \frakM(X^{\vs})=M(X^{\vs}\cup \lc\frakM(X^{\vs})\rc)$ (see Eq.~(\mref{eq:omid})). Then $q$ is either in $X\cup \{\star_1,\cdots,\star_k\}$ or is of the form $\lc \tilde{q}\rc$ with $\tilde{q}\in \overline{\frakM}^{\,\vs}(X)$.
In the first case, $q$ is of depth 0 and has been proved above. In the second case, $\tilde{q}$ has depth $n$. Note that $\phi$ and $\phi_{\vs}$ (and hence $\overline{\phi}_{\vs}$) are compatible with the brackets. So together with the induction hypothesis, we have
$$ \phi(q|_{\vec{s}})= \phi(\lc \tilde{q}\rc|_{\vec{s}}) = \phi(\lc \tilde{q}|_{\vec{s}}\rc) = \lc \phi(\tilde{q}|_{\vec{s}})\rc =\lc \overline{\phi}_{\vs}(\tilde{q})|_{\phi(\vec{s})}\rc =
(\lc \overline{\phi}_{\vs}(\tilde{q})\rc )|_{\phi(\vec{s})}
= (\overline{\phi}_{\vs}(\lc\tilde{q}\rc) )|_{\phi(\vec{s})} = \overline{\phi}_{\vs}(q)|_{\phi(\vec{s})}.$$
This completes the inductive proof of the claim.
\end{proof}

Then Item~(\mref{item:bmphi}) is a special case of the claim since the restriction of $\overline{\phi}_{\vs}$ to $\frakM^{\vs}(X)$ is $\phi_{\vs}$.

\smallskip

\noindent
(\mref{item:mbphi})
Denote $\vec{u}=(u_1,\cdots, u_k)$ and $\phi^{-1}(\vec{u})=(\phi^{-1}(u_1),\cdots,\phi^{-1}(u_k))$. By Item~(\mref{item:bmphi}) and the bijectivity of $\phi$ and $\phi_{\vs}$, we have
$$ \phi^{-1}(p|_{\vec{u}})
= \phi^{-1}(\phi_{\vs}(\phi^{-1}_{\vs}(p)) |_{\phi(\phi^{-1}(\vec{u}))})
=\phi^{-1}\left( \phi\left(\phi^{-1}_{\vs}(p)|_{\phi^{-1}(\vec{u})}\right)\right) =\phi^{-1}_{\vs}(p)|_{\phi^{-1}(\vec{u})}.$$
This is what we need.
\end{proof}

By Proposition~\mref{prop:bmrel}, we immediately have

\begin{coro}
Let $f$ be in $\frakM(X)$. Then $(s,q)$ is a \pla of $s$ in $f$
if and only if $(\phi(s),\phi_\star(q))$ is a \pla of $\phi(s)$ in $\phi(f)\in\mathcal{W}(X)$.
\mlabel{coro:eqrel}
\end{coro}
\begin{proof}
The pair $(s,q)$ is a \pla of $s$ in $f$ if and only if $f=q|_s$, which holds if and only if  $\phi(f)=\phi_\star(q)|_{\phi(s)}$, which holds if and only if $(\phi(s),\phi_\star(q))$ is a \pla of $\phi(s)$ in $\phi(f)$.
\end{proof}
We also have
\begin{coro}
Let $q_1,q_2$ and $q$ be $\vs$-bracketed words in $\frakM^{\star}(X)$.  Then $q_1=q_2|_q$ if and only if  $\phi_{\star}(q_1)=
\phi_{\star}(q_2)|_{\phi_{\star}(q)}$ in $\mathcal{W}^{\star}(X)$.
\mlabel{co:stphi}
\end{coro}
\begin{proof}
To be precise, denote the isomorphism and the bijection in Proposition~\mref{prop:bmrel}, in the case of $k=1$, by
$$\phi:=\phi_X:\frakM(X)\rightarrow \mathcal{W}(X), \quad \phi_{\star_1}:=\phi_{X,\star_1}: \frakM^{\star_1}(X)\rightarrow \mathcal{W}^{\star_1}(X).$$
Also denote $X':=X\cup \{\star\}$ where $\star$ is a symbol not in $X\cup\{\star_1\}$. Then
$\phi_{X'}|_{\frakM^{\star}(X)}=\phi_{X,\star}$.

Let $q_1, q_2, q\in \frakM^\star(X)$ be given. Define $q_2':=q_2|_{\star\to \star_1}\in \frakM^{\star_1}(X) \subseteq \frakM^{\star_1}(X')$. We also have
$q_1, q\in \frakM^\star(X)\subseteq \frakM(X')$.
Then $q_1=q_2|_q$ means $q_1=q_2'|_q$ in $\frakM(X')$. Then applying Proposition~\mref{prop:bmrel}.(\mref{item:bmphi}) to the set $X'$ with $k=1$, we have

$$\phi_\star(q_1)=\phi_{X,\star}(q_1)=\phi_{X'}(q_1) =\phi_{X'}(q_2'|_q) =\phi_{X',\star_1}(q_2')|_{\phi_{X'}(q)} =\phi_{\star_1}(q_2')|_{\phi_\star(q)}.$$
But $q_2'|_\star =q_2$, so $\phi_{\star_1}(q_2')|_{\phi_\star(q)} =\phi_{\star}(q_2)|_{\phi_\star(q)}.$
Hence
$$\phi_\star(q_1) =\phi_{\star}(q_2)|_{\phi_{\star}(q)}.$$

Conversely, if $\phi_\star(q_1)=\phi_\star(q_2)|_{\phi_\star(q)}$, then we have $\phi_\star(q_1)=\phi_{\star_1}(q_2')|_{\phi_\star(q)}$, where $q_2'=q_2|_{\star_1}.$
Then with the notations as above, applying Proposition~\mref{prop:bmrel}.(\mref{item:mbphi}), we have

$$ q_1 = \phi^{-1}_{\star}(\phi_\star(q_1)) = \phi^{-1}_{\star}(\phi_{\star_1}(q'_2)|_{\phi(q)}) =\phi^{-1}_{\star_1}(\phi_{\star_1}(q'_2))|_{\phi_\star^{-1}(\phi_\star(q))} =q'_2|_q=q_2|_q.$$
\end{proof}

\begin{theorem}
Let $f$ be a bracketed word in $\frakM(X)$. Let $(s_1,q_1)$ and $(s_2,q_2)$ be  \plas in $f$. Then
\begin{enumerate}
\item
the pairs $(s_1,q_1)$ and $(s_2,q_2)$ are separated in $f$ if and only if the pairs $(\phi(s_1), \phi_{\star}(q_1))$ and $(\phi(s_2),\phi_{\star}(q_2))$ are separated in $\phi(f)$;
\mlabel{item:sep}
\item
the pairs $(s_1,q_1)$ and $(s_2,q_2)$ are nested in $f$ if and only if the pairs $(\phi(s_1),\phi_{\star}(q_1))$ and $(\phi(s_2),\phi_{\star}(q_2))$ are nested  in $\phi(f)$;
\mlabel{item:nes}
\item
the pairs $(s_1,q_1)$ and $(s_2,q_2)$ are intersecting in $f$ if  and  only if the pairs $(\phi(s_1),\phi_{\star}(q_1))$ and $(\phi(s_2),\phi_{\star}(q_2))$ are intersecting in $\phi(f)$.
\mlabel{item:int}
 \end{enumerate}
\mlabel{thm:main2}
\end{theorem}

\begin{proof}
Let $(s_1,q_1)$ and $(s_2,q_2)$ be  \plas in $f$.
Then by Corollary~\mref{coro:eqrel}, $(\phi(s_1),\phi_{\star}(q_1))$ and $(\phi(s_2),\phi_{\star}(q_2))$ are \plas in $\phi(f)$.
\smallskip

(\mref{item:sep})
The \plas $(s_1,q_1)$ and  $(s_2,q_2)$ in $f$ are separated if and only if there exists an element $q$ in $\frakM^{\star_1,\star_2}(X)$ such that $$q_1|_{\star_1}=q|_{\star_1,\ s_2}, \quad q_2|_{\star_2}=q|_{s_1,\ \star_2}, \quad f=q|_{s_1,s_2}.$$
By Proposition~\mref{prop:bmrel}, this is so if and only if
\begin{equation}
\phi_{\star}(q_1)|_{\star_1} =\phi_{\star_1,\star_2}(q)|_{\star_1,\ \phi(s_2)}, \quad \phi_{\star}(q_2)|_{\star_2} =\phi_{\star_1,\star_2}(q)|_{\phi(s_1),\ \star_2}, \quad \phi(f)=\phi_{\star_1,\star_2}(q)|_{\phi(s_1),\phi(s_2)}.
\mlabel{eq:sep1}
\end{equation}
If this is true, then $(\phi(s_1),\phi_{\star}(q_1))$ and $(\phi(s_2),\phi_{\star}(q_2))$ are separated in $\phi(f)$.
Conversely, if the \plas $(\phi(s_1),\phi_{\star}(q_1))$ and $(\phi(s_2),\phi_{\star}(q_2))$ are separated in $\phi(f)$, then there exists an element $p$ in $\calw^{\star_1,\star_2}(X)$ such that
$$\phi_{\star}(q_1)|_{\star_1}=p|_{\star_1,\ \phi(s_2)}, \quad \phi_{\star}(q_2)|_{\star_2}=p|_{\phi(s_1),\ \star_2}, \quad \phi(f)=p|_{\phi(s_1),\phi(s_2)}.$$
Since $\phi_{\star_1,\star_2}:\frakM^{\star_1,\star_2}(X)\to \calw^{\star_1,\star_2}(X)$ is bijective, there is $q\in \frakM^{\star_1,\star_2}(X)$ such that $\phi_{\star_1,\star_2}(q)=p$. Thus Eq.~(\mref{eq:sep1}) holds and hence $(s_1,q_1)$ and  $(s_2,q_2)$ are separated in $f$.
\smallskip

\noindent
(\mref{item:nes})
The \plas $(s_1,q_1)$ and $(s_2,q_2)$ are nested in $f$ if and only if there exists an element $q$ in $\frakM^{\star}(X)$ such that $q_1=q_2|_q$ or $q_2=q_1|_q$. By Corollary~\mref{co:stphi}, this holds if and only if
\begin{equation}
\phi_\star(q_1)=\phi_{\star}(q_2)|_{\phi_{\star}(q)} \quad \text{ or } \quad \phi_{\star}(q_2)=\phi_{\star}(q_1)|_{\phi_{\star}(q)}.
\mlabel{eq:nes1}
\end{equation}
If this is true, then $(\phi(s_1),\phi_\star(q_1))$ and $(\phi(s_2),\phi_\star(q_2))$ are nested. Conversely, if
$(\phi(s_1),\phi_\star(q_1))$ and $(\phi(s_2),\phi_\star(q_2))$ are nested, then there is $p\in \calw^{\star}(X)$ such that
$$\phi_\star(q_1)=\phi_{\star}(q_2)|_{p} \quad \text{ or } \quad \phi_{\star}(q_2)=\phi_{\star}(q_1)|_{p}.$$
Since $\phi_{\star}:\frakM^{\star}(X)\to \calw^{\star}(X)$ is bijective, there is $q\in \frakM^{\star}(X)$ such that $\phi_{\star}(q)=p$. Thus Eq.~(\mref{eq:nes1}) holds and hence $(s_1,q_1)$ and  $(s_2,q_2)$ are nested in $f$.

\smallskip
\noindent
(\mref{item:int})
The \plas $(s_1,q_1)$ and $(s_2,q_2)$ are intersecting in $f$ if and only if there exist $q$ in $\frakM^{\star}(X)$ and $a,b,c$ in $\frakM(X)\backslash\{1\}$ such that
$$f=q|_{abc}, \quad q_1=q|_{\star c}, \quad q_2=q|_{a\star},$$
or
$$f=q|_{abc}, \quad q_1=q|_{a\star}, \quad q_2=q|_{\star c}.$$
By Theorem~\mref{thm:mbiso}, Proposition~\mref{prop:bmrel} and Corollary~\mref{co:stphi}, this is so if and only if
\begin{equation}
\phi(f)=\phi_\star(q)|_{\phi(a)\phi(b)\phi(c)}, \quad \phi_\star(q_1)=\phi_\star(q)|_{\phi(a)\star},\quad \phi_\star(q_2)=\phi_\star(q)|_{\star \phi(c)},
\mlabel{eq:int1}
\end{equation}
or
\begin{equation}
\phi(f)=\phi_\star(q)|_{\phi(a)\phi(b)\phi(c)}, \quad \phi_\star(q_1)=\phi_\star(q)|_{\star\phi(c)},\quad \phi_\star(q_2)=\phi_\star(q)|_{\phi(a)\star}.
\mlabel{eq:int2}
\end{equation}
If this is true, then $(\phi(s_1),\phi_\star(q_1))$ and $(\phi(s_2),\phi_\star(q_2))$ are intersecting since $\phi(a), \phi(b), \phi(c)\neq 1$.

Conversely, if $(\phi(s_1),\phi_\star(q_1))$ and $(\phi(s_2),\phi_\star(q_2))$ are intersecting in $\phi(f)$, then there exist $p$ in $\calw^{\star}(X)$ and $\alpha, \beta, \gamma$ in $\calw(X)\backslash\{1\}$ such that
$$\phi(f)=p|_{\alpha\beta\gamma}, \quad \phi_\star(q_1)=\phi_\star(q)|_{\alpha\star}, \quad \phi_\star(q_2)=\phi_\star(q)|_{\star\gamma},$$
or
$$\phi(f)=p|_{\alpha\beta\gamma}, \quad \phi_\star(q_1)=\phi_\star(q)|_{\star\gamma}, \quad \phi_\star(q_2)=\phi_\star(q)|_{\alpha\star}.$$
By the bijectivity of $\phi$ and $\phi_\star$, there are $q\in \frakM^{\star}(X)$ and $a, b, c\in \frakM(X)\backslash\{1\}$ such that $p=\phi_\star(q)$ and $\alpha=\phi(a), \beta=\phi(b), \gamma=\phi(c)$. Then Eqs.~(\mref{eq:int1}) or (\mref{eq:int2}) hold. This shows that $(s_1,q_1)$ and $(s_2,q_2)$ are intersecting in $f$.
\end{proof}

\subsection{Relative locations of bracketed subwords}
\mlabel{subsec:pla}

Now we are ready to prove our main theorem on the classification of relative locations of two bracketed subwords (\plas) in a bracketed word.

\begin{theorem} {\bf $($Main Theorem$)$}
Let $f$ be a bracketed word in $\frakM(X)$. For any two \plas $(s_1, q_1)$ and $(s_2,q_2)$ in $f$, exactly one of the following is true:
\begin{enumerate}
\item
$(s_1, q_1)$ and $(s_2,q_2)$ are separated;
\item
$(s_1, q_1)$ and $(s_2,q_2)$ are nested;
\item
$(s_1, q_1)$ and $(s_2,q_2)$ are intersecting.
\end{enumerate}
\end{theorem}
\begin{proof}
By Theorem~\mref{thm:relw}, the statement of the theorem holds when  $(s_1,q_1)$ and $(s_2,q_2)$ are replaced by the two \plas $(\phi(s_1),\phi_\star(q_1))$ and $(\phi(s_2),\phi_\star(q_2))$ in the word $\phi(f)\in \calw(X)\subseteq M(X\cup\{\lm,\rtm\})$. Then by Theorem~\mref{thm:main2}, the statement holds for $(s_1,q_1)$ and $(s_2,q_2)$.
\end{proof}

\noindent {\bf Acknowledgements}: L.~Guo acknowledges support from NSF grant DMS 1001855. S. Zheng thanks support from NSFC grant 11201201 and Fundamental
Research Funds for the Central Universities lzujbky-2013-8.


\end{document}